\documentclass[10pt]{amsart}
\usepackage{setspace}
\usepackage{amsmath, amsfonts, amsthm, amstext, xspace}
\usepackage{amssymb,latexsym}
\usepackage{xcolor}
\usepackage{hyperref}
\usepackage{comment}
\definecolor{seagreen}{RGB}{46,139,87}
\definecolor{maroon}{RGB}{128,0,0}
\definecolor{darkviolet}{RGB}{148,0,211}
\definecolor{twelve}{RGB}{100,100,170}
\definecolor{thirteen}{RGB}{100,150,50}
\definecolor{fourteen}{RGB}{200,0,0}
\definecolor{fifteen}{RGB}{0,200,0}
\definecolor{sixteen}{RGB}{0,0,200}
\definecolor{seventeen}{RGB}{200,0,200}
\definecolor{eighteen}{RGB}{0,200,200}

\usepackage{mathrsfs}

\usepackage[all]{xy}
\newtheorem{thm}{Theorem}[section]
\newtheorem*{thm*}{Theorem}
\newtheorem{lem}[thm]{Lemma}
\newtheorem{cor}[thm]{Corollary}

\newtheorem{prop}[thm]{Proposition}

\newtheorem{notn}[thm]{Notation}

\theoremstyle{definition}
\newtheorem{defin}[thm]{Definition}

\newtheorem{exm}[thm]{Example}
\theoremstyle{remark}
\newtheorem{rem2}[thm]{Remark}

\def\c{\mathbb{C}}
\def\f{\mathbb{F}}

\def\m{\mathbb{M}}

\def\q{\mathbb{Q}}
\def\r{\mathbb{R}}

\def\z{\mathbb{Z}}

\def\cf{\mathcal{F}}

\def\cp{\mathcal{P}}

\def\holim{\underset{\longleftarrow}{\operatorname{holim}}}

\def\Spec{\operatorname{Spec}}

\def\lim{\operatorname{lim}}

\newcommand{\chara}{\operatorname{char}}

\usepackage{amssymb}
\usepackage{tikz-cd}

\theoremstyle{theoremstyle}

\newtheorem*{theorem*}{Theorem}

\newtheorem*{proposition*}{Proposition}

\newtheorem*{corollary*}{Corollary}

\newtheorem{remark*}{Remark}

\newtheorem{defn*}{Definition}
\theoremstyle{definition}

\theoremstyle{theoremstyle}

\author{J.D. Quigley}
\title[Motivic Mahowald invariants over general base fields]{Motivic Mahowald invariants over general base fields}

\begin{document}
\maketitle

\begin{abstract}
The motivic Mahowald invariant was introduced in \cite{Qui19a} and \cite{Qui19b} to study periodicity in the $\mathbb{C}$- and $\mathbb{R}$-motivic stable stems. In this paper, we define the motivic Mahowald invariant over any field $F$ of characteristic not two and use it to study periodicity in the $F$-motivic stable stems. In particular, we construct lifts of some of Adams' classical $v_1$-periodic families \cite{Ada66} and identify them as the motivic Mahowald invariants of powers of $2+\rho \eta$. 
\end{abstract}

\tableofcontents

\section{Introduction} 

The Mahowald invariant 
$$M : \pi_*(S^0) \rightsquigarrow \pi_{*}(S^0), \quad \alpha \mapsto M(\alpha)$$ is a construction for producing nontrivial classes in the stable homotopy groups of spheres from classes in lower stable stems. The chromatic filtration organizes the stable stems into $v_n$-periodic families \cite{Rav84}. These $v_n$-families are completely understood when $n\leq 1$ and fairly well-understood when $n \leq 2$, but they are much more mysterious for larger $n$. Computations of  Sadofsky \cite{Sad92}, Mahowald and Ravenel \cite{MR93}, Bruner \cite{Bru98}, Behrens \cite{Beh06} \cite{Beh07}, and the author \cite{Qui19c} suggest that the Mahowald invariant of a $v_n$-periodic class is often $v_{n+1}$-periodic. Therefore the Mahowald invariant gives a (conjectural) means of studying mysterious $v_n$-periodic families, $n \geq 1$, by relating them to less mysterious $v_{n-1}$-periodic families. 

The motivating example of this phenomenon comes from the pioneering work of Mahowald and Ravenel \cite{MR93}. Adams constructed $v_1$-periodic elements $\{v_1^{4i} \eta^j, v_1^{4i}8\sigma : 1 \leq j \leq 3, i \geq 0\}$ in \cite{Ada66}. These can be constructed using the non-nilpotent $v_1$-self-map on the mod two Moore spectrum, iterated Toda brackets, and connective real K-theory. We review their definition in Section \ref{Section:Prime}, but for now we emphasize that their nontriviality depended on geometric input related to connective real K-theory. Mahowald and Ravenel gave another construction of these $v_1$-periodic elements using the Mahowald invariant. More precisely, they showed that for all $i \geq 0$, 
$$M(2^{4i+j}) \ni \begin{cases} v_1^{4i}\eta^j \quad & \text{ if } 1 \leq j \leq 3, \\
v^{4(i-1)}_1 8\sigma \quad & \text{ if } j=0.
\end{cases}$$
Although this computation also uses connective real K-theory, we regard this as an independent construction of Adams' $v_1$-periodic families since it does not appeal to the existence of a non-nilpotent $v_1$-self-map on the mod two Moore spectrum. 

Periodic phenomena in the motivic stable stems is not well-understood; we briefly recall what is known in the $\c$-motivic case. Work of Levine allows one to lift classical $v_n$-periodic families into the $\c$-motivic stable stems. Subsequent work of Andrews \cite{And18}, Gheorghe \cite{Ghe17b}, and Krause \cite{Kra18} shows that in addition to these $v_n$-periodic lifts, there exist ``exotic periodic families"  in the $\c$-motivic stable stems. For example, Andrews constructed exotic $w_1$-periodic families of elements analogous to Adams' $v_1$-periodic families mentioned above. Despite these interesting results, we still lack a complete understanding of periodic phenomena in the $\c$-motivic stable stems. Over other fields of characteristic zero, even less is known. We refer the reader to \cite[Sec. 4-5]{IO18} for an extensive discussion of the $\r$- and $\c$-motivic stable stems. 

The $\c$-motivic Mahowald invariant 
$$M^\c : \pi_{**}^\c(S^{0,0}) \to \pi_{**}^\c(S^{0,0}), \quad \alpha \mapsto M^\c(\alpha)$$
was introduced in \cite{Qui19a} to study these families over $\Spec(\c)$. In particular, it was shown that 
$$M^\c(2^{4i+j}) \ni \begin{cases} v_1^{4i}\eta \quad & \text{ if } j=1, \\
 v^{4i}_1 \tau \eta^j \quad & \text{ if } j = 2,3, \\
v^{4(i-1)}_1 8\sigma \quad & \text{ if } j=0.
\end{cases}$$
and
$$
M^\c(\eta^{4i+j}) \ni \begin{cases} w_1^{4i} \nu \quad & \text{ if } j=1, \\
w^{4i}_1 \nu^j \quad & \text{ if } j=2,3, \\
w^{4(i-1)}_1 \eta^2 \eta_4 \quad & \text{ if } j=0.
\end{cases}$$
In \cite{Qui19b}, $\r$-motivic lifts of these $v_1$- and $w_1$-periodic elements were constructed and identified as $\r$-motivic Mahowald invariants. The goal of this paper is to obtain analogous computations over general base fields of characteristic not two.

Very little is currently known about periodicity in the motivic stable stems over a general base field $F$ of characteristic not two. Work of Morel \cite{Mor12} shows that $\pi^F_{m,n}(S^{0,0}) = 0 $ for $m<n$, and the $0$-line $\bigoplus_{n \in \z} \pi^F_{n,n}(S^{0,0})$ is Milnor-Witt K-theory $K^{MW}_*(F)$. R{\"o}ndigs-Spitzweck-{\O}stv{\ae}r  computed the $1$-line $\bigoplus_{n \in \z} \pi^F_{n+1,n}(S^{0,0})$ in \cite{RSO19}, and they showed that there is a short exact sequence of Nisnevich sheaves
$$0 \to K^M_{2-n}(-)/24 \to \pi_{n+1,n}^{(-)}(S^{0,0}) \to \pi_{n+1,n}^{(-)} f_0(KQ)$$
where $K^M$ is Milnor K-theory and $f_0(KQ)$ is the effective cover of Hermitian K-theory. Other infinite computations have been done after inverting the Hopf map $\eta \in \pi_{1,1}^F(S^{0,0})$ by Guillou-Isaksen \cite{GI15}\cite{GI16}, Andrews-Miller  \cite{AM17}, R{\"o}ndigs \cite{Ron16}, and Wilson \cite{Wil18}. We refer the reader to \cite[Sec. 6]{IO18} for a survey of other computations over general base fields. 

Our first main theorem proves that many of Adams' $v_1$-periodic families can be lifted to the $F$-motivic stable stems, where $F$ is an arbitrary field of characteristic not two. The first part of the theorem applies in the case $\chara(F) > 2$. Recall that over $\Spec(\c)$, one has $v_1$-periodic families of the form $\{ v^{4i}_1 \eta,  v^{4i}_1 \tau \eta^2,  v^{4i}_1 \tau \eta^2, v^{4i}_1 8 \sigma \}_{i \geq 0}$ which can be constructed as iterated Toda brackets using \cite{Qui19a}. Moreover, the Betti realization \cite{MV99} of these classes recovers Adams' classical elements \cite{Ada66} with the same names. Let $g : F \to \overline{F}$ be the inclusion of $F$ into its algebraic closure and note that $\pi_{s,w}^{\overline{F}}(S^{0,0}) \cong \pi_{s,w}^{\c}(S^{0,0})$ when $s \geq w \geq 0$ or $s<w$ by \cite{WO17}. 

\begin{thm*}[Theorem \ref{Thm:FamiliesGeneral}, Part (1)]
Suppose that $\chara(F) > 2$. Let $\alpha \in \{ v^{4i}_1 \eta,  v^{4i}_1 \tau \eta^2,  v^{4i}_1 \tau \eta^3, v^{4i}_1 8 \sigma \}_{i \geq 0} \subset \pi_{**}^{\overline{F}}(S^{0,0})$. There exists a nontrivial class $\tilde{\alpha} \in \pi_{**}^F(S^{0,0})$ such that $g^*(\tilde{\alpha}) = \alpha$. 
\end{thm*}

Perioidicity is more subtle in characteristic zero. Although one has motivic analogs of the classical periodicity operators and $v^4_1$-periodic families over $\Spec(\c)$ \cite[Sec. 5]{Qui19a}, some of these are not well-defined over $\Spec(\r)$ or $\Spec(\q)$. 

\begin{exm}
The following three phenomena occur working over $\Spec(\r)$ but not over $\Spec(\c)$.
\begin{enumerate}
\item The class $16\sigma \in \pi_{7,4}^\r(S^{0,0})$ is nonzero \cite[Fig. 3]{DI16a}. In particular, the Toda bracket $v_1^4(-) = \langle \sigma, 16, \alpha \rangle$ cannot be defined. 
\item Recall from \cite[Thm. 5.12]{Qui19a} that $8\sigma \in M^{\c}(2^4) \subset \pi_{7,4}^\c(S^{0,0})$. As a consequence of the previous fact, one can show that $8\sigma \notin M^\r((2+\rho\eta)^4) \subset \pi_{7,4}^\r(S^{0,0})$ but $16\sigma \in M^\r((2+\rho\eta)^4) \subset \pi_{7,4}^\r(S^{0,0})$. This suggests that the $F$-motivic Mahowald invariant ``detects" differences in the order of $(2+\rho\eta)$-torsion in the image of a conjectural $F$-motivic $J$-homomorphism.
\item The class $Ph_1 \in Ext^{5,14,5}_{A^\c}(\m_2^\c,\m_2^\c)$ detects $v_1^4 \eta \in \pi_{9,5}^\c(S^{0,0})$, but it is not in the image of base change along $\r \to \c$ by \cite[Lem. 5.7]{DI16a}. In particular, the class $Ph_1$ does not survive in the $\rho$-Bockstein spectral sequence which calculates $Ext^{***}_{A^\r}(\m_2^\r,\m_2^\r)$ from $Ext^{***}_{A^\c}(\m_2^\c,\m_2^\c)$. 
\end{enumerate}
\end{exm}

Despite these obstructions to defining $v_1$-periodic families in characteristic zero, we are able to construct two nontrivial infinite families. Let $g : F \to \overline{F}$ be the inclusion of $F$ into its algebraic closure. 

\begin{thm*}[Theorem \ref{Thm:FamiliesGeneral}, Part (2)]
Suppose that $\chara(F) = 0$. Let $\alpha \in \{ v^{4i}_1 \tau \eta^2, v^{4i}_1 \tau \eta^3\}_{i \geq 0} \subset \pi_{**}^{\overline{F}}(S^{0,0})$. There exists a nontrivial class $\tilde{\alpha} \in \pi_{**}^F(S^{0,0})$ such that $g^*(\tilde{\alpha}) = \alpha$. 
\end{thm*}

Note that since $\tau \eta^4 = 0$ in $\pi_{**}^{\bar{F}}(S^{0,0})$, these classes cannot be seen using Wilson's $\eta$-local computations over the rationals \cite{Wil18}. Therefore the theorem provides two new infinite families in the $F$-motivic stable stems for any field $F$ of characteristic zero. 

In this work, we extend the definition of the motivic Mahowald invariant over any field $F$ of characteristic not two. Let $M^F(\alpha)$ denote the $F$-motivic Mahowald invariant of a class $\alpha$ in the $F$-motivic stable stems. In \cite{Qui19a} and \cite{Qui19b}, we showed that the infinite families above are realized as the $\c$-motivic Mahowald invariants of $2^i$, $i \geq 1$, and the $\r$-motivic Mahowald invariants of $(2+\rho \eta)^i$, $i \equiv 2,3 \mod 4$. Our second main theorem is that this is true over any field of characteristic not two. 

\begin{thm*}[Theorem \ref{Thm:MotMI2i}]
Suppose that $\chara(F) > 2$. Then the $F$-motivic Mahowald invariant of $(2+\rho \eta)^i$ is given by
$$
M^F((2+\rho\eta)^{4i+j}) \ni \begin{cases}
v^{4i}_1 \eta \quad & \text{ if } j =1, \\
v^{4i}_1 \tau \eta^2 \quad & \text{ if } j=2,\\
v^{4i}_1 \tau \eta^3 \quad & \text{ if } j=3, \\
v^{4i}_1 8\sigma \quad & \text{ if } j=4.
\end{cases}
$$
Suppose that $\chara(F) = 0$. Then the $F$-motivic Mahowald invariant of $(2+\rho \eta)^i$ is given by
$$
M^F((2 + \rho \eta)^{4i+j} \ni \begin{cases}
v^{4i}_1 \tau \eta^2 \quad & \text{ if } j=2,\\
v^{4i}_1 \tau \eta^3 \quad & \text{ if } j=3. 
\end{cases}
$$
\end{thm*}

\subsection{Outline}
In Section \ref{Section:Prelim}, we recall results from forthcoming work of Gepner-Heller \cite{GH18} and Gregersen-Heller-Kylling-Rognes-{\O}stv{\ae}r \cite{GHKRO18}. In particular, we discuss equivariant motivic homotopy theory, a motivic analog of Lin's Theorem, and the compatibility of both of these with base-change. We then extend the definition of the motivic Mahowald invariant to base fields $F$ of characteristic not two. We also prove the main technical lemma which is used in Section \ref{Section:General} to infer $F$-motivic Mahowald invariants from our computations over $\Spec(\f_q)$ and $\Spec(\q)$ (from Section \ref{Section:Prime}), and $\Spec(\c)$ (from \cite{Qui19a}). 

In Section \ref{Section:Prime}, we discuss $v_1$-periodicity over algebraically closed fields. We then define $v_1$-periodic families over $\Spec(\f_q)$ (where $q$ is an odd prime), $\Spec(\q_\nu)$ (where $\nu$ is any prime), and $\Spec(\q)$. Our primary tool is the $\rho$-Bockstein spectral sequence introduced by Hill in \cite{Hil11}. We also prove that the map $(2 + \rho \eta)^4$ is null on certain stunted motivic lens spectra using the Atiyah-Hirzebruch spectral sequence and base-change.

In Section \ref{Section:General}, we use our computations from Section \ref{Section:FiniteFields} to compute $M^{\f_q}(2^i)$, $i \geq 1$, and we use our computations from Section \ref{Section:Rationals} to compute $M^\q((2+\rho \eta)^i)$ for $i \equiv 2,3\mod 4$. In both cases, we follow the proof from \cite[Sec. 5]{Qui19b} which is modified from the proof of \cite[Thm. 2.17]{MR93}. We then use the key comparison lemma from Section \ref{Section:Prelim} to compute the motivic Mahowald invariants of $(2+\rho\eta)^i$ over any field $F$ of characteristic not two. The key point is $F$ fits into a sequence of field extensions
\begin{align*}
\f_q \to F \to \overline{F}, \quad &  \text{ if } \chara(F)=q \neq 0, \\
\q \to F \to \overline{F}, \quad & \text{ if } \chara(F) =0.
\end{align*}
The motivic Mahowald invariants of $(2+\rho \eta)^i$ agree over $\Spec(\f_q)$, $\Spec(\q)$, and $\Spec(\overline{F})$ in the congruence classes of $i$ where we can calculate them, so by compatibility with base-change, they must agree over $\Spec(F)$ as well. 

\subsection{Notation}
We employ the following notation and conventions throughout:
\begin{enumerate}
\item $k$, $F$ , and $L$ are fields of characteristic not two. 
\item $\f_q$ is the finite field with $q$ elements, where $q$ is an odd prime.
\item $SH(k)$ is the motivic stable homotopy category over $\Spec(k)$.
\item $S^{0,0}$ is the motivic sphere spectrum.
\item Everything is implicitly $(2,\eta)$-complete. 
\item $\m_2^k$ is the mod two motivic cohomology of a point over $\Spec(k)$.
\item $A^k$ (resp. $A^k_*$) is the motivic (resp. dual motivic) Steenrod algebra over $\Spec(k)$.
\item $Ext^{s,f,w}_{A^k}$ denotes the cohomology of the $k$-motivic Steenrod algebra in stem $s$, Adams filtration $f$, and motivic weight $w$. 
\end{enumerate}

\subsection{Acknowledgements}
The author thanks Jonas Irgens Kylling first and foremost. Many of the ideas and insights in this paper were discovered in collaboration with him during the author's visit to the University of Oslo in August 2018 and the subsequent months. The author also thanks Tom Bachmann, Mark Behrens, Jeremiah Heller, Dan Isaksen, Paul Arne {\O}stv{\ae}r, and an anonymous referee for helpful discussions. This project was also partially completed at the Newton Institute workshop on equivariant and motivic homotopy theory in 2018. We gratefully thank the Newton Institute and the University of Oslo for their support. The author was partially supported by NSF grant DMS-1547292.

\section{Motivic Lin's Theorem and the motivic Mahowald invariant revisited}\label{Section:Prelim}
In this section, let $f : k \to F$ be a field extension and let $f^* : SH(k) \to SH(F)$ be the corresponding base-change functor. 

\subsection{Motivic Lin's Theorem revisited}\label{Section:MotTate}

We begin by defining geometric universal spaces and geometric classifying spaces following \cite{GH18}. Let $\cp(Sm_k^G)$ be the category of motivic $G$-spaces over $\Spec(k)$ \cite{GH18}.

\begin{defin}[{\cite[Def. 3.2]{GH18}}]
Let $\cf$ be a family of subgroups of $G$. The \emph{universal motivic $\cf$-space over $\Spec(k)$} is the object $E_{gm}\cf_k \in \cp(Sm_k^G)$ whose value on $X \in Sm^G_k$ is
$$E_{gm}\cf_k(X) = \begin{cases}
\emptyset \quad &\text{ if } X^H \neq \emptyset \text{ for some } H \notin \cf, \\
pt \quad&\text{ else.}
\end{cases}
$$

When $\cf$ is the family of proper subgroups of $G$, we will use the notation $E_{gm}G := E_{gm}\cf$. In this case, we define
$$B_{gm}G := (E_{gm}G) / G.$$
In this case, $B_{gm}G$ is the geometric classifying space originally defined by Morel-Voevodsky \cite{MV99} and Totaro \cite{Tot99}. 
\end{defin}

Geometric universal spaces are preserved under base-change by \cite[Prop. 3.8]{GH18} and satisfy motivic analogs of many useful properties of universal spaces in classical homotopy theory. 

\begin{defin}[{\cite{Gre12}}]\label{Def:Proj}
For all $n \in \z$, the \emph{stunted motivic lens spectrum} $\underline{L}^\infty_n$ is defined by setting
$$\underline{L}^\infty_n := Th(n\gamma \to B_{gm}C_2)$$
where $\gamma$ is the tautological line bundle over $B_{gm}C_2$. Define
$$\underline{L}^\infty_{-\infty} := \underset{n}{\holim} \underline{L}^\infty_{-n}.$$
\end{defin}

The key property of this construction is the following, which was proven by Gregersen over fields of characteristic zero with finite virtual cohomological dimension, and by Gregersen-Heller-Kylling-Rognes-{\O}stv{\ae}r over fields of characteristic not two. 

\begin{thm}[{\cite[Thm. 2.0.2]{Gre12}\cite{GHKRO18}}]\label{Thm:MotLinsThm}
The map
$$S^{-1,0} \to \underline{L}^\infty_{-\infty}$$
induces a $\pi_{**}$-isomorphism after $(2,\eta)$-completion. 
\end{thm}

We will use this in the next section to define the $k$-motivic Mahowald invariant. We conclude this section by recording a useful property of  stunted motivic lens spectra. 

\begin{lem}[{\cite{GHKRO18}}]\label{Lem:ProjBaseChange}
Let $n \in \z$. Then $\underline{L}^\infty_n$ is preserved under base-change. 
\end{lem}

\subsection{The motivic Mahowald invariant revisited}\label{Section:MotMI}

Using the results of the previous section, we can extend the definition of the motivic Mahowald invariant from \cite[Sec. 2]{Qui19a} to general base fields. 

\begin{defin}~\label{Def:MotMI} 
Let $\alpha \in \pi^k_{s,t}(S^{0,0})$. We define the \emph{$k$-motivic Mahowald invariant} of $\alpha$, denoted $M^k(\alpha)$, as follows. Consider the coset of completions of the following diagram
\[
\begin{tikzcd}
S^{s,t} \arrow[dashed,rr] \arrow{d}{\alpha} && S^{-2N+1,-N} \vee S^{-2N+2,-N+1} \arrow{d} \\
S^{0,0} \arrow{r}{\simeq} & \Sigma^{1,0} \underline{L}^\infty_{-\infty} \arrow{r} & \Sigma^{1,0} \underline{L}^\infty_{-N}
\end{tikzcd}
\]
where $N>0$ is minimal so that the left-hand composition is nontrivial. If the composition of the dashed arrow with the projection onto the higher dimensional sphere is nontrivial, we define the $k$-motivic Mahowald invariant $M^k(\alpha)$ to be the coset of completions composed with the projection onto the higher dimensional sphere. Otherwise, the composition of the dashed arrow with the projection onto the higher dimensional sphere is trivial and we define the motivic Mahowald invariant $M^k(\alpha)$ to be the coset of completions composed with the projection onto the lower dimensional sphere. We illustrate this convention in the examples later in this section.
\end{defin}

Recall that the ``Squeeze Lemmas" from \cite[Sec. 3]{Qui19b} were used to compute generalized Mahowald invariants by comparing them under functors such as equivariant Betti realization and geometric fixed points. Since $\underline{L}^\infty_{-N}$ is preserved by base-change, we obtain the following comparison lemma which will be essential in the last section.

\begin{lem}\label{Lem:Squeeze}
Let $f : k \to F$. Suppose $\alpha,\beta \in \pi^{F}_{**}(S^{0,0})$ such that $\beta \in M^F(\alpha)$. Suppose further that there exist $\alpha', \beta' \in \pi^k_{**}(S^{0,0})$ such that $f^*(\alpha') = \alpha$ and $f^*(\beta') = \beta$. Then 
$$|M^k(\alpha')| \leq |\beta'|.$$
Further, if $|M^k(\alpha')| = |\beta'|$, then $\beta' \in M^k(\alpha')$. 
\end{lem}

\begin{proof}
Suppose that the $F$-motivic Mahowald invariant $M^F(\alpha)$ is defined by the commutative diagram
\[
\begin{tikzcd}
S^{s,t}  \arrow{d}{\alpha} \arrow[rr,dashed,"\beta"] && S^{-2N+1,-N} \wedge S^{-2N+2,-N+1} \arrow{d} \\
S^{0,0} \arrow{r}{\simeq} & \Sigma^{1,0} \underline{L}^\infty_{-\infty} \arrow{r} & \Sigma^{1,0} \underline{L}^\infty_{-N},
\end{tikzcd}
\]
so in particular $N>0$ is minimal so that the left-hand composite is nontrivial. Then the left-hand composite in the commutative diagram
\[
\begin{tikzcd}
S^{s,t}  \arrow{d}{\alpha'} \arrow[rr,dashed,"\beta'"] && S^{-2N+1,-N} \wedge S^{-2N+2,-N+1} \arrow{d} \\
S^{0,0} \arrow{r}{\simeq} &  \Sigma^{1,0} \underline{L}^\infty_{-\infty} \arrow{r} & \Sigma^{1,0} \underline{L}^\infty_{-N}
\end{tikzcd}
\]
must also be nontrivial since the first commutative diagram can be obtained from this one by applying $f^*$. Since there may be some $N' < N$ such that the left-hand composite is nontrivial, we only obtain an inequality as in the statement of the lemma. However, if $|\beta'| = |M^k(\alpha')|$, then $N$ is minimal and $\beta' \in M^k(\alpha')$ by definition, which proves the last claim. 
\end{proof}

Applying the previous lemma to the a composite of base-change functors gives the following.

\begin{cor}
Let $k \overset{f}{\to} F \overset{g}{\to} L$ be a sequence of field extensions. Suppose that we have $\alpha,\beta \in \pi^k_{**}(S^{0,0})$, $\alpha', \beta' \in \pi^F_{**}(S^{0,0})$, and $\alpha'', \beta'' \in \pi^L_{**}(S^{0,0})$ such that $g^*(\alpha') = \alpha''$, $g^*(\beta') = \beta''$, $f^*(\alpha) = \alpha'$,  $f^*(\beta) = \beta'$, $\beta \in M^k(\alpha)$, and $\beta'' \in M^L(\alpha'')$. Then $\beta' \in M^F(\alpha')$. 
\end{cor}

\section{$v_1$-periodic families over prime fields}\label{Section:Prime}

Over $\Spec(\c)$, one can construct $v_1$-periodic families via iterated Toda brackets following \cite{Qui19a}. We begin this section by constructing analogous families over any algebraically closed field using a result of Wilson-{\O}stv{\ae}r \cite{WO17}. We then construct lifts of these infinite families in $\pi_{**}^F(S^{0,0})$ where $F = \f_q$ ($q$ an odd prime), $F = \q_\nu$ ($\nu$ any prime), and $F = \q$. That is, we construct families which base-change to the $v_1$-periodic families in the algebraic closure. We also study the $F$-motivic homotopy of certain stunted motivic lens spectra. 

\subsection{$v_1$-periodicity over algebraically closed fields}\label{Section:Closed}

In this section, we define some infinite families over the algebraically closed fields. We begin by recalling the analogous families from the classical and $\c$-motivic settings.

In the classical setting, Adams constructed infinite families $\{v^{4i}_1 \eta^j, v^{4i}_1 8\sigma: 1 \leq j \leq 3, i \geq 0\}$ as the (nontrivial) composites
\begin{align*}
S^{j} \overset{\widetilde{\eta^j}}{\longrightarrow} S^{0}/2 \overset{v^{4i}_1}{\longrightarrow} \Sigma^{-8i} S^{0}/2 \to S^{-8i+1}, \\
S^{0} \hookrightarrow S^{0}/2 \overset{v^{4i}_1}{\longrightarrow} \Sigma^{-8i} S^{0}/2 \to S^{-8i+1},
\end{align*}
where $S^{0}/2$ is the mod two Moore spectrum, $\widetilde{\eta^j} : S^j \to S^0/2$ is a lift of $\eta^j \in \pi_j(S^0)$ to the top cell of $S^0/2$, and $v^{4}_1 : S^{0,0}/2 \to \Sigma^{-8} S^{0,0}/2$ is a non-nilpotent self-map of $S^0/2$ \cite{Ada66}. These classes are detected by the classes $\{P^i h_1^j, P^i h_0^3 h_3: 1 \leq j \leq 3, i \geq 0\}$ in the Adams spectral sequence, where $P^i(-)$ is the Massey product $P(-) := \langle h_3, h_0^4, - \rangle$ iterated $i$-times. 

We discussed $\c$-motivic lifts of these infinite families in \cite{Qui19a}. In particular, one can define $P(-) := \langle h_3, h_0^4, - \rangle$ in the $\c$-motivic Adams spectral sequence and define infinite families $\{v^{4i}_1 \eta, v^{4i}_1 \tau \eta^2, v^{4i}_1 \tau \eta^3, v^{4i}_1 8\sigma : i \geq 0\}$ as the classes detected by $\{P^i h_1, P^i \tau h_1^2, P^i \tau h_1^3, P^i h_0^3 h_3: i \geq 0\}$. All of these classes are nontrivial since they are permanent cycles (for degree reasons) and their Betti realizations are Adams' classical families. 

We can construct analogs of these classes over arbitrary algebraically closed fields of characteristic not two using the following theorem of Wilson and Wilson-{\O}stv{\ae}r:

\begin{thm}\label{Thm:WO}
Let $\overline{F}$ be an algebraically closed field of exponential characteristic $q \neq 2$. For all $s \geq w \geq 0$, there are isomorphisms $\pi^{\overline{F}}_{s,w}(S^{0,0})[q^{-1}] \cong \pi^\c_{s,w}(S^{0,0})[q^{-1}]$. 
\end{thm}

\begin{proof}
If $q > 2$, this follows from \cite[Thm. 1.1]{WO17}. If $q = 0$, this follows from the proof of \cite[Prop. 7]{Wil18}.
\end{proof}

\begin{defin}
We define $v^{4i}_1 \eta$, $v^{4i}_1 \tau \eta^2$, $v^{4i}_1 \tau \eta^3$, and $v^{4i}_1 8 \sigma$ to be the classes in $\pi_{**}^{\overline{F}}(S^{0,0})$ corresponding to the classes with the same names in $\pi_{**}^\c(S^{0,0})$ under the isomorphism in Theorem \ref{Thm:WO}. 
\end{defin}

Our primary goal in the remainder of this section is to construct lifts of these to prime fields. 

\subsection{Computations over finite fields of prime order}\label{Section:FiniteFields}

We start by constructing some infinite families in $\pi_{**}^{\f_q}(S^{0,0})$ using the $\rho$-Bockstein spectral sequence \cite{Hil11} and base-change along $\f_q \to \overline{\f}_q$. We break the analysis into two cases (Lemmas \ref{Lem:Extq1} and \ref{Lem:Extq3}) depending on the congruence class of $q$ modulo four; the results are summarized in Theorem \ref{Thm:FamiliesFiniteFields} for future reference. We then study the $\f_q$-motivic homotopy of a certain stunted lens spectrum. The computations in this section serve as a warm-up for the analogous computations over $\Spec(\q_\nu)$ ($\nu$ any prime) in Section \ref{Section:pAdicRationals} and \ref{Section:Rationals}. They will also be used in Section \ref{Section:PrimeMI}. 

We note that the $\rho$-Bockstein spectral sequence converging to $\pi_{**}^{\f_q}(kq)$ was studied by Kylling \cite{Kyl15} and the $\rho$-Bockstein spectral sequence for $\pi_{**}^{\f_q}(S^{0,0})$ was studied by Wilson \cite{Wil16} and Wilson-{\O}stv{\ae}r \cite{WO17}. We refer the reader to their work, as well as \cite{Hil11} and \cite{GHIR17}, for further applications of the $\rho$-Bockstein spectral sequence. 

\begin{defin}
The \emph{Milnor-Witt degree} of a class $\alpha \in \pi_{s,w}^k(S^{0,0})$ is defined to be $MW(x) := s-w$.  
\end{defin}

Recall that in \cite[Sec. 5]{Qui19b}, we constructed classes $v^{4i}_1 \tau \eta^j \in \pi_{**}^\r(S^{0,0})$ for all $i \geq 0$ and $2 \leq j \leq 3$. The class $v^{4i}_1 \tau \eta^j$ is (by definition) detected by $P^i \tau h_1^j \in Ext^{***}_{A^\r}$ where $P(-)$ is the matric Massey product
$$P(x) := \left\langle \begin{bmatrix} h_3 & \rho^3 h_1^2 \end{bmatrix}, \begin{bmatrix} h_0^4 \\  c_0 \end{bmatrix}, x \right\rangle.$$
Note that if $\rho^3 = 0$ in $\m_2^k$, then this operator simplifies to the Massey product $\langle h_3, h_0^4, - \rangle$ studied in \cite{Qui19a} which lifted the periodicity operator introduced by Adams in \cite{Ada66b}. 

The following lemma allows us to construct analogous classes in $Ext^{***}_{A^F}$ where $F$ is any field of characteristic not two. We refer the reader to \cite[Sec. 2.1]{IO18} for a review of Milnor K-theory $K^M_*(F)$. 

\begin{lem}[{\cite[Lem. 4.1]{KW18}}]\label{Lem:KW}
Let $Ext^i$ denote the $i$-th $Ext$-group. Over a field of characteristic not two, we have for each $i$ an extension
$$0 \to Ext^i_{A^\r} \otimes_{\f_2[\rho]} K^M_*(F)/2 \to Ext^i_{A^F} \to Tor_1^{\f_2[\rho]}(Ext^{i+1}_{A^\r}, K^M_*(F)/2) \to 0.$$
\end{lem}

In particular, there is an injective map
$$\phi_F : Ext^i_{A^\r} \otimes_{\f_2[\rho]} K_*^M(F)/2 \to Ext^i_{A^F}$$
for any field $F$ of characteristic not two. 

We recall the following calculation of the Milnor K-theory of finite fields for the reader's convenience.

\begin{thm}[{\cite{Mil70}\cite[Ex. 2.6]{IO18}}]
The Milnor K-theory of a finite field $\f_q$ is given by 
$$K^M_*(\f_q) \cong \z[u]/u^2$$
where $u = [a]$ is the class of any generator $a \in \f_q^\times \cong K_1^M(\f_q)$. In particular, we have
$$K^M_*(\f_q)/2 \cong \begin{cases}
\f_2[u]/u^2 \quad & \text{ if } q \equiv 1 \mod 4, \\
\f_2[\rho]/\rho^2 \quad & \text{ if } q \equiv 3 \mod 4,
\end{cases}
$$
where $\rho = [-1]$. 
\end{thm}

\begin{lem}\label{Lem:Extq1}
The following statements hold over $\Spec(\f_q)$ where $q \equiv 1 \mod 4$:
\begin{enumerate}
\item For all $i \geq 0$, the classes $P^i h_1$, $P^i \tau h_1^2$, $P^i \tau h_1^3$, and $P^i h_0^3 h_3$ are well-defined classes in $Ext_{\f_q}$.
\item The classes above are permanent cycles in the $\f_q$-motivic Adams spectral sequence.
\item The classes above detect nontrivial classes in $\pi_{**}^{\f_q}(S^{0,0})$ which base-change to the classes with the same name in $\pi_{**}^{\overline{\f}_q}(S^{0,0})$. 
\end{enumerate} 
\end{lem}

\begin{proof}
\begin{enumerate}
\item We have $Ext_{\f_q} \cong Ext_\c \otimes E(u)$ where $|u| = (0,-1,-1)$ by \cite[Prop. 7.1]{WO17}. These classes are just the images under $\phi_F$ of the classes in $Ext_\c$ with the same name tensored with $1$.
\item Adams differentials preserve motivic weight and decrease stem by one. It follows from \cite[Prop. 5.4]{Qui19a} that there are no possible targets for Adams differentials on these classes.
\item The statement about base-change is clear since $\pi_{**}^{\overline{\f}_q}(S^{0,0}) \cong \pi_{**}^\c(S^{0,0})$ and base-change sends clases in $\pi_{**}^{\f_q}(S^{0,0})$ to classes with the same name in $\pi_{**}^{\overline{\f}_q}(S^{0,0})$. Since base-change induces a map of Adams spectral sequences and the classes with the same name in the $\overline{\f}_q$-motivic Adams spectral sequence are nonzero, it follows from Part (2) that the classes we are interested in detect nonzero classes in the $\f_q$-motivic Adams spectral sequence. 
\end{enumerate}
\end{proof}

\begin{lem}\label{Lem:Extq3}
The following statements hold over $\Spec(\f_q)$ where $q \equiv 3 \mod 4$:
\begin{enumerate}
\item For all $i \geq 0$, the classes $P^i h_1$, $P^i \tau h_1^2$, $P^i \tau h_1^3$, and $P^i h_0^3 h_3$ are permanent cycles in the $\rho$-Bockstein spectral sequence. 
\item There are no classes in higher $\rho$-Bockstein filtration which contribute to the same tridegrees of $Ext_{\f_q}$ as the classes above.
\item The classes above are permanent cycles in the $\f_q$-motivic Adams spectral sequence.
\item The classes above detect nontrivial classes in $\pi_{**}^{\f_q}(S^{0,0})$ which base-change to the classes with the same name in $\pi_{**}^{\overline{\f}_q}(S^{0,0})$. 
\end{enumerate}
\end{lem}

\begin{proof}

We will implicitly use the fact that $\rho^2=0$ in $\m_2^{\f_q}$ in this proof. In particular, this implies that the $\rho$-Bockstein spectral sequence collapses at $E_2$ so we only need to consider $d_1$-differentials. 

\begin{enumerate}
\item We provide the proof for $P^i h_1$; the other cases are similar. The case $i = 0$ and $i=1$ follows from \cite{WilCharts}. In general, we have $MW(P^i h_1) = 4i$, $stem(P^i h_1) = 8i+1$, and $filt(P^i h_1) = 4i+1$. The possible targets of a $\rho$-Bockstein differential have the form $\rho z$ where $z \in Ext^{8i+1, 4i+2, 4i+2}_{A^\c}$, but this group is zero by \cite[Prop. 5.4]{Qui19a}. 

\item For $P^i h_1$, $P^i \tau h_1^2$, and $P^i h_0^3 h_3$, this is clear from inspection of $Ext_\c$ in \cite{Isa14b} along with \cite[Prop. 5.4]{Qui19a}. On the other hand, the class $\rho P^i h_0 h_2$ in the $E_1$-page of the $\rho$-Bockstein spectral sequence converging to $Ext_{\f_q}$ could contribute to the same tridegree as $P^i \tau h_1^2$. However, we have $d_1(P^i \tau h_2) = \rho P^i h_0 h_2$ for all $i \geq 0$, so this does not occur.

\item This follows from \cite[Prop. 5.4]{Qui19a} by similar arguments since Adams differentials preserve motivic weight and decrease stem by one. 

\item The base-change statement is clear from Part (1). Since base-change induces a map of Adams spectral sequences and the classes with the same name in the $\overline{\f}_q$-motivic Adams spectral sequence are nonzero, it follows from Part (2) that the classes we are interested in detect nonzero classes in the $\f_q$-motivic Adams spectral sequence. 
\end{enumerate}
\end{proof}

We summarize the key results from the previous two lemmas in the following definition and theorem. 

\begin{defin}
We define $v^{4i}_1 \eta$, $v^{4i}_1 \tau \eta^2$, $v^{4i}_1 \tau \eta^3$, and $v^{4i}_1 8\sigma$ to be the classes in $\pi_{**}^{\f_q}(S^{0,0})$ detected by $P^i h_1$, $P^i \tau h_1^2$, $P^i \tau h_1^3$, and $P^i h_0^3 h_3$, respectively. 
\end{defin}

\begin{thm}\label{Thm:FamiliesFiniteFields}
The classes $v^{4i}_1 \eta$, $v^{4i} \tau \eta^2$, $v^{4i}_1 \tau \eta^3$, and $v^{4i}_1 8\sigma$ in $\pi_{**}^{\f_q}(S^{0,0})$ are nonzero. 
\end{thm}

\begin{prop}\label{Prop:16NullFinite}
Let $L_m$ be the subcomplex of $\underline{L}^\infty_{-\infty}$ with cells in dimension $-5+8m \leq d \leq 2+8m$. Over $\f_q$, the degree $16$ map is null on $L_m$ for all $ m \in \z$. 
\end{prop}

\begin{rem2}
The subcomplex $L_0$ may be constructed as follows; the cases $m \neq 0$ can be obtained by suspension. Let $X = \underline{L}_{-3}^\infty = Th(-3\gamma \to B_{gm}C_2)$ and let $Y$ denote the simplicial $2$-skeleton of $X$, so $Y$ has cells in simplicial dimensions $-6 \leq d \leq 2$. Then $L_0$ is the cofiber of the inclusion of the $-6$-cell into $Y$. A cell diagram for $L_0$ can be produced from \cite[Fig. 1]{Qui19a} by restricting to the values $-5 \leq d \leq 2$ along the horizontal axis. 
\end{rem2}

\begin{rem2}
We will freely use the computation of $\pi_{**}^{\f_q}(S^{0,0})$ in low dimensions in the following proof. More precisely, we need to know $\pi_{-2k \mp \epsilon, -k \mp \epsilon}^{\f_q}$ for all $-3 \leq k \leq 3$ and $\epsilon \in \{0,1\}$. There are no Adams differentials in the tridegrees of the $\f_q$-motivic Adams spectral sequence computing these groups \cite[Sec. 7.3-7.4]{WO17}. With this in mind, these groups can be obtained as follows: 
\begin{enumerate}
\item When $q \equiv 1 \mod 4$, the relevant motivic homotopy groups satisfy $\pi_{**}^{\f_q}(S^{0,0}) \cong \pi_{**}^\c(S^{0,0})\{1\} \oplus \pi_{*+1,*+1}^{\c}(S^{0,0}) \{u\}.$ This follows from the previous statement about Adams differentials and the calculation of the Adams $E_2$-page in \cite[Prop. 7.1]{WO17}. 
\item When $q \equiv 3 \mod 4$, the relevant motivic homotopy groups can be obtained from the $E_2$-page of the $\rho$-Bockstein spectral sequence depicted in \cite[Fig. 2]{DI16a} by replacing infinite $\rho$-towers $x[\rho]$ by $x[\rho]/(\rho^2)$ and inserting classes $y \tau^{2i+1}\rho$, $i \geq 0$, for any $y \in \pi_{**}^\c(S^{0,0})$ which supports nontrivial $h_0$-multiplication but is not $h_0$-divisible. This follows from the previous statement about Adams differentials and the $\rho$-Bockstein spectral sequence $d_1$-differential $d_1(\tau) = \rho h_0$ from \cite[Prop. 3.2]{DI16a}. 
\end{enumerate}
\end{rem2}

\begin{proof}
The isomorphism $\pi_{**}^{\overline{\f}_q}(S^{0,0}) \cong \pi_{**}^\c(S^{0,0})$ and \cite[Prop. 5.11]{Qui19a} imply that the result holds over $\overline{\f}_q$. Let $X := L_m \wedge D^{\f_q} L_m$ where $D^{\f_q}(-) := F(-,S^{0,0})$ is the $\f_q$-motivic Spanier-Whitehead dual functor. By the argument from the proof of \cite[Prop. 5.11]{Qui19b}, we see that the class detecting $16$ in $\pi_{0,0}^{\f_q}(X)$ must base-change to a $\tau$-torsion class in $\pi_{**}^{\overline{\f}_q}(S^{0,0})$. 

The $\f_q$-motivic homotopy group $\pi_{0,0}^{\f_q}(X)$ may be computed via the Atiyah-Hirzebruch spectral sequence arising from the filtration of $X$ by topological dimension. As in the proof of \cite[Prop. 5.11]{Qui19b}, the possible contributions to $\pi_{0,0}^{\f_q}(X)$ in the Atiyah-Hirzebruch spectral sequence have the form $\alpha[2k\pm \epsilon, k\pm \epsilon]$ with $\alpha \in \pi_{-2k \mp \epsilon, -k \mp \epsilon}^{\f_q}(S^{0,0})$ where $-3 \leq k \leq 3$ and $\epsilon \in \{0,1\}$. 

In addition to the classes $\alpha$ which base-change to zero in $\pi_{**}^{\overline{\f}_q}(S^{0,0})$, we may omit classes which appeared in the proof of \cite[Prop. 5.11]{Qui19b} since the attaching map structure of $X$ is independent of the base-field. 

When $q \equiv 1 \mod 4$, the classes
$$\alpha \in \{uh_3h_0^i, 1 \leq i \leq 3; uh_2h_0; uh_0^j, j \geq 1\}$$
lie in the correct bidegrees of $\pi_{**}^{\f_q}(S^{0,0})$, base-change to $\tau$-torsion classes (actually to zero) in $\pi_{**}^{\overline{\f}_q}(S^{0,0})$, and have not been considered in the proof of \cite[Prop. 5.11]{Qui19b}. We rule them out using the Atiyah-Hirzebruch spectral sequence.

Recall that the differentials in the Atiyah-Hirzebruch spectral sequence correspond to the attaching maps in $X$. The $d_1$-differentials correspond to attaching maps detected by $h_0$, which in turn correspond to a nontrivial action of $Sq^1$ in $H^{**}(X)$. The action of $Sq^1$ on $H^{**}(X)$ may be calculated using the Cartan formula, the known action of $Sq^1$ on $H^{**}(L_0)$ (as depicted in \cite[Fig. 1]{Qui19a}), and the known action of $Sq^1$ on $H^{**}(DL_0)$ (calculated for a Spanier-Whitehead dual using the action of $Sq^1$ on $H^{**}(X)$ and the fact that $\chi Sq^1 = Sq^1$, where $\chi$ is conjugation in the motivic Steenrod algebra). 

\begin{enumerate}
\item The classes $u h_3 h_0^i$, $1 \leq i \leq 3$, are the targets of $d_1$- and $d_8$-differentials. We verify this claim in the case $m=0$ for simplicity; the other cases follow from $8$-fold periodicity of the action of $Sq^i$ on $H^{**}(\underline{L}^\infty_{-N})$ (where $N$ is any integer) for $i \leq 4$. 

The classes $uh_3 h_0^i$, $1 \leq i \leq 3$, all lie in $\pi_{6,3}^{\f_q}(S^{0,0})$, so we must study the action of $Sq^1$, $Sq^2$, $Sq^4$, and $Sq^8$ on the generators for an additive basis of $H^{-6,-3}(X)$. Any such generator has the form $e_{i} \otimes f_{k}$ where $e_{i} \in H^{i,j}(L_0)$ and $f_{k} \in H^{k,\ell}(DL_0)$ with $i+k=-6$ and $j+\ell = -3$. The only options are then $e_{-5} \otimes f_{-1}$ and $e_{-4} \otimes f_{-2}$. We have 
$$Sq^1(e_{-5} \otimes f_{-1}) = Sq^1(e_{-5}) \otimes f_{-1} + e_{-5} \otimes Sq^1(f_{-1}) = e_{-4} \otimes f_{-1},$$
$$Sq^1(e_{-4} \otimes f_{-2}) = Sq^1(e_{-4}) \otimes f_{-2} + e_{-4} \otimes Sq^1(f_{-2}) = e_{-4} \otimes f_{-1}.$$
We may take a basis to be $\{e_{-5} \otimes f_{-1}, e_{-5} \otimes f_{-1} + e_{-4} \otimes f_{-2}\}$, so then we see that the classes $uh_0^ih_3$ on the cell $e_{-5} \otimes f_{-1}$ are targets of $d_1$-differentials. 

We claim that the remaining classes $uh_0^i h_3$ on the cell corresponding $e_{-5} \otimes f_{-1} + e_{-4} \otimes f_{-2}$ are the targets of $d_8$-differentials. To see this, we apply the Cartan formula to calculate
$$Sq^8(e_{-5} \otimes f_{-1} + e_{-4} \otimes f_{-2}) = e_{0} \otimes f_{2}.$$
We therefore have the claimed $d_8$-differential if we can show that the classes $u h_0^i$ survive on the cell corresponding to $e_0 \otimes f_2$ up to the $E_8$-page of the Atiyah-Hirzebruch spectral sequence.

\item The class $uh_2h_0$ supports a $d_1$-differential.
\item The classes $uh_0^j$, $j \geq 1$, are the targets of $d_1$-differentials. 
\end{enumerate}

When $q \equiv 3 \mod 4$, there are no classes which lie in the correct bidegrees, base-changes to a $\tau$-torsion class in $\pi_{**}^{\overline{\f}_q}(S^{0,0})$, and were not considered in the proof of \cite[Prop. 5.11]{Qui19b}.

We have ruled out any classes which could detect $16$ in $\pi_{0,0}^{\f_q}(X)$, so it must be zero. 
\end{proof}

\subsection{Computations over the $\nu$-adic rationals}\label{Section:pAdicRationals}

We now make the analogous computations over $\Spec(\q_\nu)$. The computations of this section serve as input for Section \ref{Section:Rationals}. 

The mod two Milnor K-theory of $\q_\nu$ is given below; see \cite[Pg. 10]{Wil18} for details.
\[K^M_*(\q_\nu)/2 \cong \begin{cases}
\f_2[\pi, u]/(\pi^2, u^2) \quad & \text{ if } \nu \equiv 1 \mod 4, \\
\f_2[\pi, \rho]/(\rho^2, \rho \pi + \pi^2) \quad & \text{ if } \nu \equiv 3 \mod 4, \\
\f_2[\pi,\rho,u]/(\rho^3,u^2,\pi^2,\rho u, \rho \pi, \rho^2 + u \pi) \quad & \text{ if } \nu = 2. 
\end{cases}
\]

\begin{lem}
The following statements hold over $\Spec(\q_\nu)$ for $\nu \equiv 1 \mod 4$:
\begin{enumerate}
\item For all $i \geq 0$, the classes $P^i h_1$, $P^i \tau h_1^2$, $P^i \tau h_1^3$, and $P^i h_0^3 h_3$ are well-defined classes in $Ext_{\q_\nu}$.
\item The classes above are permanent cycles in the $\q_\nu$-motivic Adams spectral sequence.
\item The classes above detect nontrivial classes in $\pi_{**}^{\q_\nu}(S^{0,0})$ which base-change to the classes with the same name in $\pi_{**}^{\overline{\q}_\nu}(S^{0,0})$. 
\end{enumerate}
\end{lem}

\begin{proof}
This follows from the same argument as Lemma \ref{Lem:Extq1}; the key point is that adjoining one more power of $u$ does not carry us out of the isomorphism range in \cite[Prop. 5.4]{Qui19a}. 
\end{proof}

\begin{lem}
The following statements hold over $\Spec(\q_\nu)$ where $\nu = 2$ or $\nu \equiv 3 \mod 4$:
\begin{enumerate}
\item For all $i \geq 0$, the classes $P^i h_1$, $P^i \tau h_1^2$, $P^i \tau h_1^3$, and $P^i h_0^3 h_3$ are permanent cycles in the $\rho$-Bockstein spectral sequence. 
\item There are no classes in higher $\rho$-Bockstein filtration which contribute to the same tridegrees of $Ext_{\q_\nu}$ as the classes above.
\item The classes above are permanent cycles in the $\q_\nu$-motivic Adams spectral sequence.
\item The classes above detect nontrivial classes in $\pi_{**}^{\q_\nu}(S^{0,0})$ which base-change to the classes with the same name in $\pi_{**}^{\overline{\q}_\nu}(S^{0,0})$. 
\end{enumerate}
\end{lem}

\begin{proof}
This follows from the same argument as Lemma \ref{Lem:Extq3}. 
\end{proof}

\begin{defin}
Let $\nu$ be any prime. We define $v^{4i}_1 \eta$, $v^{4i}_1 \tau \eta^2$, $v^{4i}_1 \tau \eta^3$, and $v^{4i}_1 8\sigma$ to be the classes in $\pi_{**}^{\q_\nu}(S^{0,0})$ detected by $P^ih_1$, $P^i \tau h_1^2$, $P^i \tau h_1^3$, and $P^i h_0^3 h_3$, respectively. 
\end{defin}

\begin{thm}\label{Thm:FamiliespAdic}
The classes $v^{4i}_1 \eta$, $v^{4i} \tau \eta^2$, $v^{4i}_1 \tau \eta^3$, and $v^{4i}_1 8\sigma$ in $\pi_{**}^{\q_\nu}(S^{0,0})$ are nonzero. 
\end{thm}

\begin{proof}
A slight modification of the proof from the case $F = \f_q$ implies that the classes $P^i h_1$, $P^i \tau h_1^2$, $P^i \tau h_1^3$, and $P^i h_0^3 h_3$, $i \geq 0$, are permanent cycles in the $\q_\nu$-motivic Adams spectral sequence. The images of these classes under base-change along $f : \q_\nu \to \overline{\q}_\nu$ detect nonzero classes in the $\overline{\q}_\nu$-motivic Adams spectral sequence. Therefore they must detect nonzero classes in the ${\q_\nu}$-motivic Adams spectral sequence. 
\end{proof}

\begin{prop}\label{Prop:16NullpAdic}
Let $L_m$ be the subcomplex of $\underline{L}^\infty_{-\infty}$ with cells in dimensions $-5+8m \leq d \leq 2 + 8m$. Over $\q_\nu$, the degree $16$ map is null on $L_m$ for all $m \in \z$. 
\end{prop}

\begin{proof}
The reductions from the proof of Proposition \ref{Prop:16NullFinite} carry over \emph{mutatis mutandis}. Let $X := L_m \wedge D^{\q_\nu} L_m$. We split the analysis into three cases.

When $\nu \equiv 1 \mod 4$, the classes
$$\alpha \in \{\pi u \tau h_3 h_1^2; \pi u \tau c_0 h_1; \pi u h_3 h_0^i, 1 \leq i \leq 3; \pi u h_2 h_0  \}$$
lie in the correct bidegrees of $\pi_{**}^{\q_\nu}(S^{0,0})$, base-change to $\tau$-torsion classes (actually to zero) in $\pi_{**}^{\overline{\q}_\nu}(S^{0,0})$, and have not been considered in the proof of \cite[Prop. 5.11]{Qui19b}.

\begin{enumerate}
\item The classes $\pi u \tau h_3 h_1^2$ and $\pi u \tau c_0 h_1$ are the targets of $d_2$-differentials.
\item The classes $ \pi u h_3 h_0^i$, $1 \leq i \leq 3$, are the targets of $d_1$-differentials.
\item The class $\pi u h_2 h_0$ is the target of a $d_1$-differential. 
\end{enumerate}

When $\nu \equiv 3 \mod 4$ (resp. $\nu = 2$), the only new classes are the ones appearing in the case where $\nu \equiv 1 \mod 4$, with $\pi u$ replaced by $\pi \rho$ (resp. $\pi u$ replaced by $\rho^2$). The same argument carries through to eliminate these classes.

We have ruled out any classes which could detect $16$ in $\pi_{0,0}^{\q_\nu}(X)$, so it must be zero. 
\end{proof}

\subsection{$v_1$-periodic families over the rationals}\label{Section:Rationals}
In this section, we construct the infinite families $v^{4i}_1 \tau \eta^2$ and $v^{4i}_1 \tau \eta^3$ in $\pi_{**}^\q(S^{0,0})$. We also study the $\q$-motivic homotopy of $L_m$. In both cases, we employ previous calculations along with the motivic Hasse principle developed in \cite[Sec. 4-5]{OO13}. 

The construction of the $\r$-motivic May spectral sequence \cite[Sec. 4.3]{Qui19b} may be modified by replacing $\m_2^\r$ by $\m_2^\q$ and $(A^\r)^\vee$ by $(A^\q)^\vee$ to obtain the $\q$-motivic May spectral sequence. The discussion in \cite[Sec. 5.1]{Qui19b} carries over to define a $\q$-motivic periodicity operator 
$$P_\q(x) := \left\langle \begin{bmatrix} h_3 & \rho^3 h_1^2 \end{bmatrix}, \begin{bmatrix} h_0^4 \\  c_0 \end{bmatrix}, x \right\rangle$$
on the elements $x \in Ext^{***}_{A^\q}(\m_2^\q, \m_2^\q)$ which are both $h_0^4$ and $c_0$-torsion. 

\begin{lem}
The following statements hold over $\Spec(\q)$:
\begin{enumerate}
\item For all $i \geq 0$, the classes $P^i_\q \tau h_1^2$ and $P^i_\q \tau h_1^3$ are nontrivial in $Ext^{***}_{A^\q}(\m_2^\q,\m_2^\q)$.  
\item The classes above are permanent cycles in the $\q$-motivic Adams spectral sequence.
\item The classes above detect nontrivial classes in $\pi_{**}^{\q}(S^{0,0})$ which base-change to the classes with the same name in $\pi_{**}^{\overline{\q}}(S^{0,0})$. 
\end{enumerate}
\end{lem}

\begin{proof}
The analogous statements hold over every completion $\q_\nu$ of $\q$:
\begin{itemize}
\item Over $\r$, the statements follow from \cite{Qui19b}.
\item Over $\q_\nu$ with $\nu$ a prime, the statements follow from Section \ref{Section:pAdicRationals}. The identification of the relevant elements as (matric) Massey products in $Ext^{***}_{A^{\q_\nu}}(\m_2^{\q_\nu},\m_2^{\q_\nu})$ follows from the construction of $\phi_F$ in the proof of \cite[Lem. 4.1]{KW18}. 
\end{itemize}

The map $K^M_*(\q)/2 \to \prod_\nu K^M_*(\q_\nu)/2$ induces an injective map between the $E_1$-pages of the motivic May spectral sequences converging to $Ext^{***}_{A^\q}$ and $\prod_\nu Ext^{***}_{A^{\q_\nu}}$. Then $(1)$ is clear since the same results hold over every completion of $\q$. The motivic Hasse map also induces a map of motivic Adams spectral sequences, so $(2)$ follows similarly. Finally, $(3)$ follows from quotienting by $\rho$ as in \cite[Sec. 5.1]{Qui19b}. 
\end{proof}

\begin{defin}
We define $v^{4i}_1 \tau \eta^2$ and $v^{4i}_1 \tau \eta^3$ to be the classes in $\pi_{**}^\q(S^{0,0})$ detected by $P^i \tau h_1^2$ and $P^i \tau h_1^3$, respectively. 
\end{defin}

\begin{thm}
The classes $v^{4i}_1 \tau \eta^2$ and $v^{4i}_1 \tau \eta^3$ are nonzero in $\pi_{**}^\q(S^{0,0})$. 
\end{thm}

\begin{prop}\label{Prop:16NullRationals}
Let $L_m$ be the subcomplex of $\underline{L}^\infty_{-\infty}$ with cells in dimensions $-5+8m \leq d \leq -2 + 8m$. Over $\q$, the degree $16$ map is null on $L_m$ for all $m \in \z$. 
\end{prop}

\begin{proof}
The analogous results hold over every completion $\q_\nu$ of $\q$:
\begin{enumerate}
\item Over $\r$, this follows from \cite[Sec. 5.2]{Qui19b}.
\item Over $\q_\nu$ with $\nu$ a prime, this follows from Section \ref{Section:pAdicRationals}. 
\end{enumerate}

The result then follows from base-change and the motivic Hasse map. In particular, we see that the induced map from the $E_2$-page of the $\q$-motivic Adams spectral sequenceto the product over all completions of the $E_2$-page of the $\q_\nu$-motivic Adams spectral sequence is injective in the relevant tridegrees. Since there are no differentials in this range (by base-change and the analogous statements over $\Spec(\r)$ and $\Spec(\q_\nu)$), the map in homotopy groups is injective in the relevant bidegrees. The result follows. 
\end{proof}

\section{Motivic Mahowald invariant of $(2+\rho\eta)^i$}\label{Section:General}
We now apply the computations from Section \ref{Section:Prime} to compute $M^F((2+\rho\eta)^i)$ for all $i \geq 0$ if $\chara(F) > 2$ and for all $i \equiv 2,3 \mod 4$ if $\chara(F) = 0$. 

\subsection{Prime field Mahowald invariants}\label{Section:PrimeMI}
We begin by computing over prime fields. That is, we compute the $\f_q$-Mahowald invariant of $2^i$, $i \geq 1$, where $q > 2$ is prime, and the $\q$-Mahowald invariant of $(2+\rho \eta)^i$, $i \equiv 2,3\mod 4$. 

\begin{thm}\label{Thm:InitialMotMI2i}
Let $i \geq 1$. The $\f_q$-Mahowald invariant of $2^i$ is given by
\[ M^{\f_q}(2^{4i+j}) \ni \begin{cases}
v^{4i}_1 8\sigma \quad & j=0,\\
v^{4i}_1\eta &j=1,\\
v^{4i}_1 \tau \eta^2 &j=2,\\
v^{4i}_1 \tau \eta^3 &j=3.
\end{cases}
\]
\end{thm}

\begin{proof}
The proof is analogous to the proof of \cite[Thm. 5.12]{Qui19b}; we summarize the necessary changes below.
\begin{enumerate}
\item The isomorphism $\pi_{**}^{\overline{\f}_q}(S^{0,0}) \cong \pi_{**}^\c(S^{0,0})$ implies that $M^{\overline{\f}_q}(2^{4i+j})$ has the desired form. Base-change along $\f_q \to \overline{\f}_q$ and Lemma \ref{Lem:Squeeze} then give the upper bound $|M^{\f_q}(2^{4i+j})| \leq |M^{\overline{\f}_q}(2^{4i+j})| = |M^\c(2^{4i+j})|$. Moreover, if bound is tight, then $M^{\f_q}(2^{4i+j})$ must base-change to $M^{\overline{\f}_q}(2^{4i+j})$. By Section \ref{Section:FiniteFields}, the classes in the theorem statement are the only classes which do this.
\item Proposition \ref{Prop:16NullFinite} and low-dimensional computations show that the inequality is an equality; compare with the proof of \cite[Thm. 2.17]{MR93}. 
\end{enumerate}
\end{proof}

\begin{thm}
Let $ i \geq 0$ and let $j \in \{2,3\}$. The $\q$-Mahowald invariant of $2^{4i+j}$ is given by
\[
M^\q((2+\rho \eta)^{4i+j}) \ni \begin{cases}
v^{4i}_1 \tau \eta^2 \quad & j=2, \\
v^{4i}_1 \tau \eta^3 \quad & j=3.
\end{cases}
\]
\end{thm}

\begin{proof}
We follow the same proof idea as above. 
\begin{enumerate}
\item Lemma \ref{Lem:Squeeze} applied to the isomorphism $\pi_{**}^{\overline{\q}}(S^{0,0}) \cong \pi_{**}^\c(S^{0,0})$ implies that $M^{\overline{\q}}((2+\rho\eta)^{4i+j})$ has the desired form. Base-change along $\q \to \overline{\q}$ and Lemma \ref{Lem:Squeeze} then give the upper bound $|M^{\q}(2^{4i+j})| \leq |M^{\overline{\q}}(2^{4i+j})| = |M^\c(2^{4i+j})|$. Moreover, if this inequality is actually an equality, then $M^{\q}(2^{4i+j})$ must base-change to $M^{\overline{\q}}(2^{4i+j})$; by Section \ref{Section:Rationals}, the classes in the theorem statement are the only possible classes with this property.
\item Proposition \ref{Prop:16NullRationals} and low-dimensional computations show that the inequality is an equality; compare with the proof of \cite[Thm. 2.17]{MR93}. 
\end{enumerate}
\end{proof}

\subsection{$F$-Mahowald invariants of $2^i$}\label{Section:GeneralMI}
We now apply the base-change functor, Lemma \ref{Lem:Squeeze}, and our computations over $\Spec(\f_q)$, $\Spec(\q)$, and $\Spec(\c)$ to compute the $F$-Mahowald invariants of $2^i$ for any field $F$ of characteristic not two. 

\begin{notn}
We will use the notation
$$k \overset{f}{\to} F \overset{g}{\to} L$$
to denote any of the following sequences of field extensions:
$$\f_q \to F \to \overline{F} \quad \text{ or } \quad \q \to F \to \overline{F}.$$
\end{notn}

The classes in the $F$-motivic Adams spectral sequence may be defined by comparison with the $\rho$-Bockstein spectral sequence in $SH(k)$ and $SH(L)$. 

\begin{defin} Let $F$ be any field as above. 
\begin{enumerate}
\item Suppose $\chara(F) > 2$. We define $v^{4i}_1 \eta$, $v^{4i} \tau \eta^2$, $v^{4i}_1 \tau \eta^3$, and $v^{4i}_1 8\sigma$ to be the classses in $\pi_{**}^{F}(S^{0,0})$ detected by $P^i h_1$, $P^i \tau h_1^2$, $P^i \tau h_1^3$, and $P^i h_0^3 h_3$, respectively, in the $F$-motivic Adams spectral sequence. 
\item Suppose $F$ be a field with $\chara(F) = 0$. We define $v^{4i}_1 \tau \eta^2$ and $v^{4i}_1 \tau \eta^3$ to be the classes in $\pi_{**}^F(S^{0,0})$ detected by $P^i \tau h_1^2$ and $P^i \tau h_1^3$, respectively, in the $F$-motivic Adams spectral sequence.  
\end{enumerate}
\end{defin}

\begin{thm}\label{Thm:FamiliesGeneral}
Let $F$ be any field as above. 
\begin{enumerate}
\item Suppose $\chara(F) > 2$. The classes $v^{4i}_1 \eta$, $v^{4i}_1 \tau \eta^2$, $v^{4i}_1 \tau \eta^3$, and $v^{4i}_1 8\sigma$ in $\pi_{**}^{F}(S^{0,0})$ are nonzero. 
\item Suppose $\chara(F) =0$. The classes $v^{4i}_1 \tau \eta^2$ and $v^{4i}_1 \tau \eta^3$ in $\pi_{**}^{F}(S^{0,0})$ are nonzero. 
\end{enumerate}
\end{thm}

\begin{proof}
The classes are all permanent cycles by base-change along $k \to F$ and the results of Section \ref{Section:Prime}. They are nonzero by base-change along $F \to L$ along with the isomorphism $\pi_{**}^{L}(S^{0,0}) \cong \pi_{**}^\c(S^{0,0})$ from Theorem \ref{Thm:WO}.
\end{proof}

\begin{thm}\label{Thm:MotMI2i}
Suppose that $\chara(F) > 2$. Then the $F$-motivic Mahowald invariant of $(2+\rho \eta)^i$ is given by
$$
M^F((2+\rho\eta)^{4i+j}) \ni \begin{cases}
v^{4i}_1 \eta \quad & \text{ if } j =1, \\
v^{4i}_1 \tau \eta^2 \quad & \text{ if } j=2,\\
v^{4i}_1 \tau \eta^3 \quad & \text{ if } j=3, \\
v^{4i}_1 8\sigma \quad & \text{ if } j=4.
\end{cases}
$$
Suppose that $\chara(F) = 0$. Then the $F$-motivic Mahowald invariant of $(2+\rho \eta)^i$ is given by
$$
M^F((2 + \rho \eta)^{4i+j} \ni \begin{cases}
v^{4i}_1 \tau \eta^2 \quad & \text{ if } j=2,\\
v^{4i}_1 \tau \eta^3 \quad & \text{ if } j=3. 
\end{cases}
$$
\end{thm}

\begin{proof}
Applying Lemma \ref{Lem:Squeeze} shows that 
$$|M^k(2^i)| \leq |M^F(2^i)| \leq |M^L(2^i)|.$$
Theorem \ref{Thm:InitialMotMI2i} identifies the left-hand side and Theorem \ref{Thm:WO} identifies the right-hand side with $|M^\c(2^i)|$ which was computed in \cite{Qui19a}. These sides agree; the theorem follows. 
\end{proof}

\bibliographystyle{plain}
\bibliography{master}

\end{document}